\documentclass{amsart}
\usepackage{amssymb}

\newcommand{\rsp}{\raisebox{0em}[2.7ex][1.3ex]{\rule{0em}{2ex} }}

\newcommand{\Q}{{\mathbb Q}}

\newcommand{\Z}{{\mathbb Z}}
\newcommand{\OO}{\mathcal O}
\newcommand{\fZ}{\mathfrak 2}
\newcommand{\fp}{\mathfrak p}
\newcommand{\fP}{\mathfrak P}
\newcommand{\fA}{\mathfrak A}
\newcommand{\fr}{\mathfrak r}
\newcommand{\frb}{\widetilde{\fr}}
\newcommand{\fR}{\mathfrak R}
\newcommand{\fq}{\mathfrak q}

\newcommand{\disc}{\operatorname{disc}}
\newcommand{\rank}{\operatorname{rank}}
\newcommand{\gen}{{\operatorname{gen}}}

\newcommand{\Cl}{\operatorname{Cl}}

\newcommand{\Gal}{\operatorname{Gal}}
\newcommand{\Aut}{\operatorname{Aut}}
\newcommand{\GL}{\operatorname{GL}}
\newcommand{\im}{\operatorname{im}\,}
\newcommand{\eps}{\varepsilon}
\newcommand{\la}{\langle}
\newcommand{\ra}{\rangle}
\newcommand{\lra}{\longrightarrow}

\newcommand{\ts}{\textstyle}
\newcommand{\ab}{\operatorname{ab}}

\newtheorem{thm}{Theorem}
\newtheorem{prop}[thm]{Proposition}
\newtheorem{lem}[thm]{Lemma}

\theoremstyle{definition}

\begin{document}
\title{On $2$-Class field towers of some imaginary 
	quadratic number fields}
\author{Franz Lemmermeyer}
\address{M\"orikeweg 1, 73489 Jagstzell, Germany}
\email{hb3@ix.urz.uni-heidelberg.de}

\abstract 
We construct an infinite family of imaginary quadratic number 
fields with $2$-class groups of type $(2,2,2)$ whose Hilbert
$2$-class fields are finite.
\endabstract

\maketitle

\section{Introduction}
Let $k$ be an imaginary quadratic number field. It has been 
known for quite a while that the $2$-class field tower of an 
imaginary quadratic number field is infinite if 
$\rank \Cl_2(k) \ge 5$, but it is not known how far from best 
possible this bound is. F.~Hajir \cite{H} has shown that
the $2$-class field tower is infinite if $\Cl(k) \supseteq (4,4,4)$,
and again we do not know if this result can be improved. In this 
article we will study the $2$-class field towers of a family of 
quadratic fields whose $2$-class groups have rank $3$. 

For quadratic discriminants $d$, let $\Cl_2(d)$ and $h_2(d)$
denote the $2$-class group and  the $2$-class number
of $\Q(\sqrt{d}\,)$, respectively. The fundamental unit
of $\Q(\sqrt{m}\,)$ will be denoted by $\eps_m$, whether
$m$ is a discriminant or not.
A factorization $d = d_1\cdot d_2$ of a discriminant $d$ into two 
coprime discriminants $d_1$, $d_2$ is called a $C_4$-factorization
if $(d_1/p_2) = (d_2/p_1) = + 1$ for all primes $p_1 \mid d_1$,
$p_2 \mid d_2$. A $D_4$-factorization of $d$ is a factorization
$d = (d_1\cdot d_2)\cdot d_3$ such that $d_1\cdot d_2$ is a
$C_4$-factorization. Finally, $d = d_1\cdot d_2\cdot d_3$ is called
a $H_8$-factorization if $(d_1d_2/p_3) = (d_2d_3/p_1) = 
 (d_3d_1/p_2) = +1$ for all primes $p_j \mid d_j$, $j = 1, 2, 3$. 
It is known (cf. \cite{Lem}) that $d$ admits a $G$-factorization 
($G \in \{C_4, D_4\}$)
if and only if $k = \Q(\sqrt{d}\,)$ admits an extension $K/k$ 
which is unramified outside $\infty$, normal over $\Q$, and 
which has Galois group $\Gal(K/k) \simeq G$. 

Let $F^1$ denote the $2$-class field of a number field $F$, and
put $F^2 = (F^1)^1$. Then in our case genus theory shows that
$d = \disc k = d_1d_2d_3d_4$ is the product of exactly four prime
discriminants, and moreover we have $k_\gen = 
\Q(\sqrt{d_1}, \sqrt{d_2}, \sqrt{d_3}, \sqrt{d_4}\,)$.

Our aim is to prove the following

\begin{thm}\label{main}
Let $k = \Q(\sqrt{d}\,)$ be an imaginary quadratic number field 
with discriminant $d = -4pqq'$, where $p \equiv 5 \bmod 8$,
$q \equiv 3 \bmod 8$ and $q' \equiv 7 \bmod 8$ are primes such
that $(q/p) = (q'/p) = -1$. Then $\Gamma_n = \Gal(k^2/k)$ is 
given by
$$ \begin{array}{rl}
   \Gamma_n = \la \rho, \sigma, \tau: & 
   \rho^4 = \sigma^{2^{n+1}} = \tau^4 = 1,
	\rho^2 = \sigma^{2^n}\tau^2, \\
	& [\sigma,\tau] = 1, [\rho,\sigma] = \sigma^2,
	[\rho,\tau] = \sigma^{2^n} \tau^{2} \ra, 
  \end{array}$$
where $n$ is determined by $\Cl_2(-qq') = (2,2^n)$.
Moreover, $\Cl_2(k) \simeq (2,2,2)$ and 
$\Cl_2(k^1) \simeq (2,2^n)$; for results about class and unit 
groups of subfields of $k^1/k$ see the proof below.
\end{thm}

We note a few relations which have proved to be useful for 
computing in $\Gamma_n$:
$(\rho\sigma)^2 = \rho^2$, 
$(\rho\sigma\tau)^2 = (\rho\tau)^2 = \tau^2$,
$[\rho,\tau^2] = 1$. 

\section{Preliminary Results on Quadratic Fields}

We will use a few results on $2$-class groups of quadratic
number fields which can easily be deduced from genus theory.
We put $m = pqq'$, where $p, q, q'$ satisfy the conditions
in Thm. \ref{main}, and set $k = \Q(\sqrt{-m}\,)$. 

\begin{lem}
The ideal classes of  $\fZ = (2,1 + \sqrt{-m}\,)$, 
$\fp = (p,\sqrt{-m}\,)$ and $\fq = (q,\sqrt{-m}\,)$ generate 
$\Cl_2(k) \simeq (2,2,2)$.
\end{lem}

\begin{proof}
Since $(q/p) = (q/p') = -1$ and $p \equiv 5 \bmod 8$, there are
no $C_4$-factorizations of $d$, and this shows that
$\Cl_2(k) \simeq (2,2,2)$. The ideal classes generated by
$\fZ$, $\fp$ and $\fq$ are non-trivial simply because
their norm is smaller than $m$; the same reasoning
shows that they are independent.
\end{proof}

\begin{lem}\label{L3}
The quadratic number field $\widetilde{k} = \Q(\sqrt{-qq'}\,)$ has 
$2$-class group $\Cl_2(\widetilde{k}) \simeq (2,2^n)$ for some 
$n \ge 1$; it is generated by the prime ideals  
$\fZ = (2,1+\sqrt{-qq'}\,)$ above $2$ and a
prime ideal $\widetilde{\fp}$ above $p$.  Moreover, the ideal 
classes of $\fZ$ and $\fq = (q, \sqrt{-qq'}\,)$ generate 
a subgroup $C_2$ of type $(2,2)$; we have $n \ge 2$ if and 
only if $(q'/q) = -1$, and in this case, 
the square class in $C_2$ is $[\fZ\fq]$.
\end{lem}

\begin{proof}
Since $q \equiv 3 \bmod 8$, the only possible $C_4$-factorization
is $d = -q' \cdot 4q$; this is a  $C_4$-factorization if and only
if $(-q'/q) = +1$. 
In this case, exactly one of the ideal classes $[\fZ]$,
$[\fq]$ and $[\fZ\fq]$ is a square; by genus theory,
this must be $[\fZ\fq]$.
Finally, let $\widetilde{\fp}$ denote a prime ideal above $p$. Then
$[\,\widetilde{\fp}\,]$ is no square in $\Cl_2(\widetilde{k})$ because 
$(-q/p) = -1$, $[\fZ\widetilde{\fp}\,]$ is no square because 
$(-q'/2p) = -1$, $[\fq\widetilde{\fp}\,]$ is no square because 
$(-q'/pq) = -1$, and
$[\fZ\fq\widetilde{\fp}\,]$ is no square because $(-q'/2qp) = -1$.
This implies that $[\,\widetilde{\fp}\,]$ has order $2^n$.
\end{proof}

\begin{lem}\label{pqq}
The real quadratic number field $F = \Q(\sqrt{m}\,)$
has $2$-class number $2$; the prime ideal $\fp$ above $p$
is principal, and the fundamental unit $\eps_m$ of $\OO_F$ becomes
a square in $F(\sqrt{p}\,)$.
\end{lem}

\begin{proof}
Consider the prime ideals $\fp$,
$\fq$ and $\fq'$ in $\OO_F$ above $p$, $q$ and $q'$, respectively; 
since $h_2(m) = 2$ (by genus theory) and $N\eps_m = +1$, there must 
be a relation between their ideal classes besides $\fp\fq\fq' \sim 1$.
Now clearly $\fq$ cannot be principal, since this would imply
$X^2  - m y^2 = \pm 4q$; writing $X = qx$ and dividing by
$q$ gives $qx^2 - pq'y^2 = \pm 4$. But this contradicts
our assumption that $(q/p) = -1$. By symmetry, $\fq'$ cannot be 
principal; since $h_2(F) = 2$, their
product $\fq\fq'$ is, hence we have $\fp \sim \fq\fq' \sim 1$.
The equation $X^2 - my^2 = \pm 4p$ then leads as above to
$px^2 - qq' y^2 = - 4$ (the plus sign cannot hold since
it would imply that $(p/q) = 1$), and it is easy to see
that the unit $\eta = \frac12(x\sqrt{p} + y \sqrt{qq'}\,)$ 
satisfies $\eta^2 = \eps_m^u$ for some odd integer $u$. 
\end{proof}

\section{The $2$-class field tower of $k$}

In this section we prove that the $2$-class field tower
of $k$ stops at $k^2$, and that $\Gal(k^2/k) \simeq \Gamma_n$. 

\subsection{Class Numbers of Quadratic Extensions}

Let $K/k$ be a quadratic unramified extension. Then
$q(K) = 1$; we define
\begin{itemize}
\item $\zeta = \zeta_6$ if $q = 3$ and $\zeta = -1$ otherwise; 
\item $e_j$ is the unit group of the maximal order $\OO_j$ of $k_j$;
\item $N_j$ is the relative norm of $k_j/k$;
\item $\kappa_j$ is the subgroup of ideal classes in $\Cl(k)$
      which capitulate in $k_j$;
\item $h_j$ denotes the $2$-class number of $k_j$.
\end{itemize}

Then we claim that Table \ref{T1} gives the unit group, the
relative norm of the unit group, the capitulation order, the $2$-class
number and the relative norm of the $2$-class group for the
quadratic unramified extensions $k_j/k$. 

\begin{table}[h]
\caption{}\label{T1}
\begin{center}
\begin{tabular}{|c||c|c|c|c|c|cc|} \hline
\rsp $j$ & $k_j$ & $e_j$ & $N_j e_j$ & $\# \kappa_j$ & $h_j$ &
	\multicolumn{2}{c|}{$N_j \Cl_2(k_j)$} \\ \hline \hline
\rsp $1$ & $\Q(i,\sqrt{m}\,)$ & $\la i, \eps_{m} \ra$ & $\la 1 \ra$ &
	$4$  &  $8$ & 
		\multicolumn{2}{c|}{$\la [\fZ], [\fp] \ra $} \\ \hline
\rsp $2$ & $\Q(\sqrt{-q},\sqrt{pq'}\,)$ & $\la \zeta, \eps_{pq'} \ra$ 
	& $\la 1 \ra$ & $4$  &  $8$ & 
	$\la [\fZ \fp], [\fZ \fq] \ra $ &
		$ \la [\fZ \fp], [\fq] \ra $ \\ \hline
\rsp $3$ & $\Q(\sqrt{-p},\sqrt{qq'}\,)$ & 
	$\la -1, \eps_{qq'} \ra$ & $\la 1 \ra$ &
	$4$  &  $8$ & \multicolumn{2}{c|}{$\la [\fp], [\fq] \ra $} \\ \hline
\rsp $4$ & $\Q(\sqrt{p},\sqrt{-qq'}\,)$ & 
	$\la -1, \eps_{p} \ra$ & $\la -1 \ra$ & $2$  &  $2^{n+3}$ & 
 	\multicolumn{2}{c|}{$ \la [\fp], [\fZ\fq] \ra $} \\ \hline
\rsp $5$ & $\Q(\sqrt{-q'},\sqrt{pq}\,)$ & 
	$\la -1, \eps_{pq} \ra$ & $\la 1 \ra$ &
	$4$  &  $8$ & $ \la [\fZ],[\fp\fq] \ra $ &
			$ \la [\fZ],[\fq] \ra $ \\ \hline
\rsp $6$ & $\Q(\sqrt{-pq},\sqrt{q'}\,)$ & 
	$\la -1, \eps_{q'} \ra$ & $\la 1 \ra$ &
	$4$  &  $8$ & $ \la [\fZ],[\fq] \ra $ &
			$ \la [\fZ],[\fp\fq] \ra $\\ \hline
\rsp $7$ & $\Q(\sqrt{-pq'},\sqrt{q}\,)$ & 
	$\la -1, \eps_{q} \ra$ & $\la 1 \ra$ &
	$4$  &  $8$ & $ \la [\fZ \fp], [\fq]  \ra $ &
	      $\la [\fZ \fp], [\fZ \fq]  \ra$ \\ \hline
\end{tabular} \end{center}
\end{table}
\medskip

The left hand side of the column $N_j \Cl_2(k_j)$ refers to 
case $A$ (i.e. $(q/q') = -1$), the right hand side to case $B$.
We will now verify the table. The unit groups are easy to
determine - we simply use the following proposition from
\cite{BSS}:

\begin{prop}
Let $k$ be an imaginary quadratic number field, and assume
that $K/k$ is a quadratic unramified extension. Then $K$
is a $V_4$-extension, and $q(K) = 1$.
\end{prop}

Applying this to the extension $K = k_j$ we find that
the unit group $e_j$ is generated
by the roots of unity in $k_j$ and the fundamental unit
of the real quadratic subfield. Since $\# \kappa_j =
2(E_k:N_{K/k}E_K)$ for unramified quadratic extensions $K/k$,
(see \cite{BSS}), the order of the capitulation kernel is easily
computed from the unit groups. Similarly, the $2$-class 
number of $h(k_j)$ is given by the class number formula
(see \cite{LemK}), and since the unit index equals $1$ in
all cases, its computation is trivial. 

Let us compute $N_j \Cl_2(k_j)$ in a few cases. Take e.g. 
$k_j = \Q(\sqrt{p},\sqrt{-qq'}\,)$: here the
prime ideal $\fp$ is a norm because $(-qq'/p) = + 1$,
whereas $\fZ$ and $\fq$ are inert, since
$p \equiv 5 \bmod 8$ and $(p/q) = -1$. This implies that
the ideal class $[\fZ\fq]$ must be a norm, since
$(\Cl_2(k):N_j \Cl_2(k_j)) = 2$ by class field theory.

As another example, take $k_j = \Q(\sqrt{-q}, \sqrt{pq'}\,)$
and assume that we are in case B), i.e. that $(q/q') = 1$.
Then $\fq$ is a norm since $(pq'/q) = + 1$, and the ideals
$\fZ$ and $\fp$ are inert; this yields
$N_j \Cl_2(k_j) = \la [\fq], [\fZ\fp]\ra$.

\subsection{The $2$-class group of $k_4$ is $(4,2^{n+1})$}

We know $N_4 \Cl_2(k_4) = \la [\fp], [\fZ\fq] \ra$
and $\# \kappa_4 = 2$. Since $\fp = (p,\sqrt{d}\,)$ becomes
principal, we have $\kappa_4 = \la [\fp] \ra$.

Let $\widetilde{\fp}$ denote a prime ideal above $p$ 
in $\Q(\sqrt{-qq'}\,)$; we have shown in Lemma \ref{L3} that 
$\widetilde{\fp}^{2^{n-1}} \sim \fZ\fq$. Since
this ideal class does not capitulate in $k_4$, we
see that $\widetilde{\fp}$ generates a cyclic subgroup
of order $2^n$ in $\Cl_2(k_4)$. 

Let $\OO$ denote the maximal order of $k_4$; then
$p \OO = \fP^2 {\fP'}^2$. We claim that
$\fP^2 \sim \widetilde{\fp}$. To this end, let $s$ and $t$ be the
elements of order $2$ in $\Gal(k_4/\Q)$ which fix
$F$ and $k$, respectively. Using the identity
$2 + (1 + s + t + st) = (1+s) + (1+t) + (1+st)$ of the
group ring $\Z[\Gal(k_4/\Q)]$ and observing that
$\Q$ and the fixed field of $st$ have odd class numbers
we find
$$\fP^2 \sim \fP^{1+s} \fP^{1+t} \fP^{1+st}
	\sim \widetilde{\fp} \fp \sim \widetilde{\fp},$$
where the last relation (in $\Cl_2(k_4)$) comes from
the fact that $\fp \in \kappa_4$. 

Thus $\la [\fZ], [\fP] \ra$ is a subgroup
of type $(2,2^{n+1})$ (hence of index $2$) in  $\Cl_2(k_4)$.
Taking the norm to $F$ we find 
$N_F \la [\fZ], [\fP] \ra = \la [\,\widetilde{\fp}\,] \ra$;
therefore there exists an ideal $\fA$ in $\OO$
such that $N_F \fA \sim \fZ$ (equivalence in $\Cl(F)$). 
We conclude that $\Cl_2(k_4) = \la [\fZ], 
[\fA], [\fP] \ra$. 
Similarly, letting $N_4$ denote the norm in $k_4/k$
we get $N_4\la [\fZ], [\fP] \ra = \la [\,\widetilde{\fp}\,] \ra$;
this shows that $N_4 \fA \sim \fZ\fq$ or
$N_4 \fA \sim \fZ\fq\fp$ (equivalence in $\Cl(k)$; which of the
two possibilities occurs will be determined below).
But since $\fp \sim 1 $ in $\Cl_2(k_4)$ we can conclude that
$\fA^{1+t} \sim \fZ\fq$. This gives
$\fA^2 \sim \fZ \cdot \fZ\fq
 \sim \fq$, and in particular $[\fA]$ has order $4$
in $\Cl_2(k_4)$. 

Finally we claim that $\Cl_2(k_4) = \la [\fA], [\fP] \ra$,
i.e. that $[\fZ] \in \la [\fA], [\fP] \ra$.
This is easy: from $\fP^2 \sim \widetilde{\fp}$ we deduce that
$\fP^{2^n} \sim \fZ\fq$, hence we get
$\fA^2\fP^{2^n} \sim \fZ$. 

\subsection{Proof that $k^2 = k^3$}

Since $\Cl_2(k_4) \simeq (4,2^{n+1})$, we can apply
Prop. 4 of \cite{BLS}, which says that $k_4$ has abelian
$2$-class field tower if and only if it has a quadratic
unramified extension $K/k_4$ such that $h_2(K) = \frac12
h_2(k_4)$. 

Put $M = \Q(i,\sqrt{p},\sqrt{qq'}\,)$, and consider
its subfield $M^+ = \Q(\sqrt{p},\sqrt{qq'}\,)$.
$M^+$ is the Hilbert $2$-class field of $F = \Q(\sqrt{m}\,)$,
hence $h_2(M^+) = 1$, and the class number formula gives
$q(M^+) = 2$. In fact it follows from Lemma \ref{pqq} that
$E_{M^+} = \la -1, \eps_p, \eps_{qq'}, \sqrt{\eps_m}\,\ra$.

According to Thm. 1.ii).1. of \cite{LemAA}, Hasse's unit
index $Q(M)$ equals $1$ if $w_M \equiv 4 \bmod 8$
($w_M$ denotes the number of roots of unity in $M$;
in our case $w_M = 4$ or $w_M = 12$) and if the ideal 
generated by $2$ is not a square in $M^+$ (this is the 
case here since $2 \nmid \disc M^+$). Thus $Q(M) = 1$ and 
$E_M = \la \xi, \eps_p, \eps_{qq'}, \sqrt{\eps_m}\,\ra$,
where $\xi$ is a primitive $4^{th}$ ($q \ne 3$) or 
$12^{th}$ ($q = 3$) root of unity. This shows that the 
unit index of the extension $M/\Q(i)$ equals $2$, and 
the class number formula finally gives 
$$ \ts h_2(M) = \frac14 \cdot 2 \cdot 1 \cdot 2^n \cdot 8
	= 2^{n+2} = \frac12 h_2(k_4).$$

\subsection{Computation of $\Gal(k^2/k)$}

Put $L = k^2$ and $K = k_4$; clearly 
$\sigma = \big(\frac{L/K}{\fP}\big)$ and 
$\tau = \big(\frac{L/K}{\fA}\big)$ generate
the abelian subgroup $\Gal(L/K) \simeq (2^{n+1},4)$ of
$\Gamma_n = \Gal(L/k)$. If we put
$\rho = \big(\frac{L/k}{\fZ}\big)$ then $\rho$
restricts to the nontrivial automorphism of $K/k$
(since $[\fZ]$ is not a norm from $k_4/k$), which shows 
that $\Gamma_n = \la \rho, \sigma, \tau \ra$.
The relations are easily computed:
$\rho^2 = \big(\frac{L/K}{\fZ}\big) = \sigma^{2^n}\tau^2$,
since $\fZ \sim \fA^2\fP^{2^n}$, 
$\rho^{-1}\sigma\rho =  \big(\frac{L/K}{\fP^\rho}\big) 
  = \sigma^{-1}$, since $\fP \fP^\rho = \fp \sim 1$, and
$\tau\rho^{-1}\tau\rho =  \big(\frac{L/K}{\fA^{1+\rho}}\big) 
  = \big(\frac{L/K}{\fZ\fq}\big) = \sigma^{2^n}$. Thus
$$ \begin{array}{rl}
	\Gamma_n = \la \rho, \sigma, \tau: & 
	   \rho^4 = \sigma^{2^{n+1}} = \tau^4 = 1,
	   \rho^2 = \sigma^{2^n}\tau^2,  \\
	     & [\sigma,\tau] = 1, [\rho,\sigma] = \sigma^2,
		[\tau,\rho] = \sigma^{2^{n-1}} \tau^2 \ra
	\end{array} $$
as claimed. We refer the reader to Hasse's report \cite{Has}
for the used properties of Artin symbols.

\subsection{Additional Information on Units and Class Groups}

Using the presentation of $\Gamma_n$ it is easy to
compute the abelianization of its subgroups of index $2$;
this shows at once that $\Cl_2(k_j) \simeq (2,4)$
for all $j \ne 4$ in case A and in for all $j \ne 5, 7$ in case B. 
More results are contained in Table \ref{T2}.

\begin{table}
\caption{}\label{T2}
\begin{center}
\begin{tabular}{|c||c|cc|c|c|} \hline
\rsp $j$ & $k_j$ & \multicolumn{2}{c|}{$\Cl_2(k_j)$} & $\kappa_j$ & 
	$\Gal(k^2/k_j)$ \\ \hline \hline
\rsp $1$ & $\Q(i,\sqrt{m}\,)$ & \multicolumn{2}{c|}{$(2,4)$} & 
	$\la [\fZ], [\fp] \ra $ & 
	$\la \rho, \sigma, \tau^2 \ra$ \\ \hline
\rsp $2$ & $\Q(\sqrt{-q},\sqrt{pq'}\,)$ & \multicolumn{2}{c|}{$(2,4)$} & 
	$\la [\fq], [\fZ\fp] \ra $ &
	   $\la \rho\sigma, \rho\tau, \sigma^2, \tau^2 \ra$  \\ \hline
\rsp $3$ & $\Q(\sqrt{-p},\sqrt{qq'}\,)$ & \multicolumn{2}{c|}{$(2,4)$} & 
	$\la [\fp], [\fq] \ra $ & 
    $\la \rho\tau, \sigma, \tau^2 \ra$ \\ \hline
\rsp $4$ & $\Q(\sqrt{p},\sqrt{-qq'}\,)$ &
	 \multicolumn{2}{c|}{$(4,2^{n+1})$} &	$\la [\fp] \ra $ &
    $\la \sigma, \tau \ra$ \\ \hline
\rsp $5$ & $\Q(\sqrt{-q'},\sqrt{pq}\,)$ & $(2,4)$ & $(2,2,2)$ & 
	$\la [\fp\fq], [\fZ\fp] \ra $ &
    $\la \rho, \sigma^2, \tau \ra$  \\ \hline
\rsp $6$ & $\Q(\sqrt{-pq},\sqrt{q'}\,)$ & \multicolumn{2}{c|}{$(2,4)$} & 
	$\la [\fp\fq], [\fZ] \ra $ &
    $\la \rho, \sigma\tau, \sigma^2 \ra$  \\ \hline
\rsp $7$ & $\Q(\sqrt{-pq'},\sqrt{q}\,)$ &  $(2,4)$ & $(2,2,2)$ &
	$\la [\fZ], [\fq] \ra $ & 
	$\la \rho\sigma, \tau, \sigma^2 \ra$ \\ \hline
\end{tabular}
\end{center}
\end{table}

The computation of the capitulation kernels $\kappa_j$
is no problem at all: take $k_1$, for example. We have seen
that $\fp$ is principal in $\Q(\sqrt{m}\,)$, hence
it is principal in $k_1$. Moreover, $\fZ = (1+i)$ is
clearly principal, and since we know from Table \ref{T1}
that $\# \kappa_1 = 4$ we conclude that 
$\kappa_1 = \la [\fp], [\fZ] \ra$.

Before we determine the Galois groups corresponding to 
the extensions $k_j/k$ we examine whether 
$N_4 \fA \sim \fZ\fq$ or $N_4 \fA \sim \fZ\fp\fq$ (equivalence
in $\Cl_2(k)$). We do this as follows: first we choose
a prime ideal $\fR$ in $\OO$ such that $[\fA] = [\fR]$
(this is always possible by Chebotarev's theorem). Then
its norms $\frb$ in $\OO_F$ and $\fr$ in $\OO_k$ are
prime ideals, and we have $\fZ\frb \sim 1$ in $\Cl_2(F)$.
This implies that $2r = x^2+2qq'y^2$, from which
we deduce that $(2r/q) = +1$. If we had
$\fZ\fq\fr \sim 1$ in $\Cl_2(k)$, this would imply
$2qr = U^2 + mv^2$; writing $U = qu$ this gives
$2r = qu^2 + pq' v^2$. This in turn shows $(2r/q) = (pq'/q)
 = (-q'/q)$, which is a contradiction if $(q/q') = -1$. 
Thus $N_4 \fA \sim \fZ\fp\fq$ in case A).
Similarly one shows that $N_4 \fA \sim \fZ\fq$ in case B). 

Now consider the automorphism $\tau = \big(\frac{L/K}{\fA}\big)$;
we have just shown that $\tau = \big(\frac{L/K}{\fR}\big)$.
Let $M$ be a quadratic extension of $K$ in $L$; then
the restriction of $\tau \in \Gal(L/K)$ to $M$ is
$\tau|_M = \big(\frac{M/K}{\fR}\big)$, and this is the
identity if and only if the prime ideals in the ideal class 
$[\fR]$ split in $M/K$. But from $2r = x^2+2qq'y^2$ we get
$r \equiv 3 \bmod 4$, $(2r/q) = +1$, hence $(q/r) = 1$ since
$q \equiv 3 \bmod 8$. This shows that $\fR$ splits
in $K_5 = K(\sqrt{q}\,) = \Q(\sqrt{p}, \sqrt{q}, \sqrt{-q'}\,)$,
thus $\tau$ fixes $K_5$, and we find 
$\Gal(L/K_5) = \la \tau, \sigma^2\ra$.

Next $\fP$ splits in $K_1 = K(\sqrt{-1}\,)$, hence $\sigma$ 
fixes $K_1$, and we have $\Gal(L/K_1) = \la \sigma, \tau^2 \ra$.
Finally, $\rho$ fixes those extensions $k_j/k$ in which
$[\fZ]$ splits, i.e. $k_1$, $k_5$ and $k_6$. In particular,
$\rho$ fixes their compositum $K_3$, and we find
$\Gal(L/K_3) = \la \rho, \sigma^2, \tau^2 \ra$.

Using elementary properties of Galois theory we can 
fill in the entries in the last column of Tables \ref{T2}
and \ref{T3}. 

The $2$-class group structure for
the subfields $K_j$ of $k^1/k$ of relative degree $4$,
the relative norms of their $2$-class groups and their
Galois groups are given in Table \ref{T3}. The structure
of $\Cl_2(K_j)$ is easily determined from
$\Cl_2(K_j) \simeq \Gal(k^2/K_j)^{\ab}$. Note
that in every extension $K_j/k$ the whole $2$-class group
of $k$ is capitulating. 

\begin{table}
\caption{}\label{T3}
\begin{center}
\begin{tabular}{|c||c|c|cc|c|} \hline
\rsp $j$ & $K_j$ & $ \Cl_2(K_j) $ 
	& \multicolumn{2}{c|}{$N_j \Cl_2(K_j)$} & 
	$\Gal(k^2/K_j)$ \\ \hline \hline
\rsp $1$ & $\Q(\sqrt{-1}, \sqrt{p}, \sqrt{qq'}\,)$ &
	$(2,2^{n+1})$ & \multicolumn{2}{c|}{$\la [\fp] \ra$} & 
	$\la \sigma, \tau^2 \ra$ \\ \hline
\rsp $2$ & $\Q(\sqrt{-1}, \sqrt{-q}, \sqrt{-pq'}\,)$ &
	$(2,4)$ & \multicolumn{2}{c|}{$\la [\fZ\fp] \ra$} & 
	$\la \rho\sigma, \sigma^2 \ra$ \\ \hline
\rsp $3$ & $\Q(\sqrt{-1}, \sqrt{-q'}, \sqrt{-pq}\,)$ &
	$(2,4)$ & \multicolumn{2}{c|}{$\la [\fZ] \ra$} & 
	$\la \rho, \sigma^2 \ra$ \\ \hline
\rsp $4$ & $\Q(\sqrt{p}, \sqrt{-q}, \sqrt{q'}\,)$ &
	$(2,2^{n+1})$ & $\la [\fZ\fq] \ra $ & $\la [\fZ\fp\fq] \ra $ &
	$ \la \sigma\tau, \sigma^2 \ra $  \\ \hline
\rsp $5$ & $\Q(\sqrt{p}, \sqrt{-q'}, \sqrt{q}\,)$ &
	$(4,2^n)$ & $ \la [\fZ\fp\fq] \ra $ & $ \la [\fZ\fq] \ra $ &
	$ \la \tau, \sigma^2 \ra $  \\ \hline
\rsp $6$ & $\Q(\sqrt{-q}, \sqrt{-p}, \sqrt{-q'}\,)$ &
	$(2,4)$ &  $ \la [\fp\fq] \ra $ & $ \la [\fq] \ra $ &
	$ \la \rho\tau, \sigma^2 \ra $ \\ \hline
\rsp $7$ & $\Q(\sqrt{q'}, \sqrt{-p}, \sqrt{q}\,)$ &
	$(2,4)$ &  $ \la [\fq] \ra $ & $ \la [\fp\fq] \ra $ &
	$ \la \rho\sigma\tau, \sigma^2 \ra $ \\ \hline
\end{tabular}
\end{center}
\end{table}

It is also possible to determine the unit groups $E_j$ 
of the fields $K_j$ (see Table \ref{T4}). In fact, applying
the class number formula to the $V_4$-extensions $K_j/k$
we can determine the unit index $q(K_j/k)$ since we already
know the class numbers. Thus all we have to do is check
that the square roots of certain units lie in the
field $K_j$. This is done as follows: consider the quadratic
number field $M = \Q(\sqrt{q}\,)$. Since it has odd class number
and since $2$ is ramified in $M$ we conclude that there exist
integers $x, y \in \Z$ such that $x^2 - qy^2 = \pm 2$; from
$q \equiv 3 \bmod 8$ we deduce that, in fact, $x^2 - qy^2 = -2$.
Put $\eta_q = (x+y\sqrt{q}\,)/(1+i)$; then $\eta_q^2 = -\eps_q^u$ for
some odd integer $u$ shows that $i\eps_q$ becomes a square
in $M(i)$. Doing the same for the prime $q'$ we also see
that $\eta_q \eta^{-1}_{q'} \in \Q(\sqrt{q}, \sqrt{q'}\,)$,
and this implies that 
$\sqrt{\eps_q\eps_{q'}} \in \Q(\sqrt{q}, \sqrt{q'}\,)$. The other
entries in Table \ref{T4} are proved similarly.

Using the same methods we can actually show that
$$ E = \la \zeta, \eps_p, \sqrt{i\eps_q}, \sqrt{i\eps_{q'}}, 
   \sqrt{\eps_{qq'}}, \sqrt{i\eps_{pq}}, \sqrt{i\eps_{pq'}}, 
   \sqrt{\eps_m}\, \ra $$
has index $2$ in the full unit group of $k_\gen$. 

\begin{table}
\caption{}\label{T4}
\begin{center}
\begin{tabular}{|c||c|c|} \hline
\rsp $j$ & $E_j$ & $q(K_j/k)$ \\ \hline \hline
\rsp $1$ & $\la i, \eps_p, \eps_{qq'}, 
		\sqrt{\eps_m\eps_{pq'}}\, \ra $ & $2$ \\ \hline
\rsp $2$ & $\la \zeta, \sqrt{i\eps_q}, \eps_{pq'}, 
			\sqrt{\eps_m\eps_{pq'}}\, \ra $ & $4$ \\ \hline
\rsp $3$ & $\la i, \sqrt{i\eps_{q'}}, \eps_{pq}, 
			\sqrt{\eps_m\eps_{pq}}\, \ra $ & $4$ \\ \hline
\rsp $4$ & $\la -1, \eps_p, \eps_{q'}, 
			\sqrt{\eps_{q'}\eps_{pq'}}\, \ra $ & $2$ \\ \hline
\rsp $5$ & $\la -1, \eps_p, \eps_{q}, 
			\sqrt{\eps_{q}\eps_{pq}}\, \ra $ & $2$ \\ \hline
\rsp $6$ & $\la -1, \eps_{pq}, \sqrt{\eps_{pq}\eps_{pq'}}, 
			\sqrt{\eps_{qq'}} \,\ra $  & $4$ \\ \hline
\rsp $7$ & $\la -1, \eps_q, \sqrt{\eps_q\eps_{q'}}, 
			\sqrt{\eps_{qq'}} \,\ra $  & $4$ \\ \hline
\end{tabular}
\end{center}
\end{table}

\section{The Field with Discriminant $d = -420$}\label{Secl}

The smallest example in our family of quadratic number fields
is given by $d = -420 = -4\cdot 3 \cdot 5 \cdot 7$. Poitou
\cite{Poi} noticed that the class field tower of $k = \Q(\sqrt{d}\,)$
must be finite, and Martinet \cite{Mar} observed (without proof) that 
its class field tower terminates at the second step with $k^2$, and
that $(k^2:k) = 32$. In this section, we will give a complete
proof (this will be used in \cite{Yam}).

The unit group of the genus class field $k_\gen$ is
$$E = \la \zeta_{12}, \eps_5, \sqrt{i\eps_3}, \sqrt{i\eps_7}, 
      \sqrt{i\eps_{15}}, \sqrt{i\eps_{35}}, \sqrt{\eps_{21}}, \eta \ra,$$
where $\eta = \sqrt{\eps_7^{-1} \eps_5\sqrt{\eps_3\eps_7\eps_{15}\eps_{35}
\eps_{105}}\,} = \frac12(1+2\sqrt5 + \sqrt7 + \sqrt{15} + \sqrt{21}\,)$.
Since $k_\gen$ has class number $4$, the second Hilbert class field 
is just its $2$-class field, which can be constructed explicitly: 
$k^2 = k^1(\sqrt{\mu},\sqrt{\nu}\,)$, where
$\mu = (4i-\sqrt{5}\,)(2+\sqrt5\,)$ and 
$\nu = (2\sqrt{-5} + \sqrt7\,)(8+3\sqrt{7}\,)$.

We claim that $k^2$ has class number $1$. Odlyzko's unconditional
bounds show that $h(k^2) \le 10$, and since we know that 
it has odd class number, it suffices to prove that no odd
prime $\le 10$ divides $h(k^2)$. For $p = 5$ and $p = 7$
we can use the following result which goes back to Gr\"un
\cite{Gru34}: 

\begin{prop}
Let $L/k$ be a normal extension of number fields, and let $K$ denote
the maximal subfield of $L$ which is abelian over $k$. 
\begin{enumerate}
\item[i)] If $\Cl(L/K)$ is cyclic, then $h(L/K) \mid (L:K)$;
\item[ii)] If $\Cl(L)$ is cyclic, then $h(L) \mid (L:K)e$, where 
      $e$ denotes the exponent of
      $N_{L/K} \Cl(L)$. Observe that $e \mid h(K)$, and that $e = 1$ 
      if $L$ contains the Hilbert class field $K^1$ of $K$. 
\end{enumerate}
Similar results hold for the $p$-Sylow subgroups.
\end{prop}

\begin{proof}
Let $C$ be a cyclic group of order $h$ on which $\Gamma = \Gal(L/k)$
acts. This is equivalent to the existence of a homomorphism
$\Phi: \Gamma \lra \Aut(C) \simeq \Z/(h-1)\Z$. Since $\im \Phi$ 
is abelian, $\Gamma' \subseteq \ker \Phi$, hence $\Gamma'$ acts 
trivially on $C$. Now $\Gamma'$ corresponds to the field $K$ via 
Galois theory, and we find $ N_{L/K}c = c^{(L:K)}.$
Putting $C = \Cl(L/K)$, we see at once that $(L:K)$ annihilates 
$\Cl(L/K)$; if we denote the exponent of $N_{L/K} \Cl(L)$ by $e$
and put $C = \Cl(L)$, then we find in a similar way that $(L:K)e$ 
annihilates $\Cl(L)$. 
\end{proof}

This shows at once that $h(k^2)$ is not divisible by $5$ or $7$.
For $p = 3$, the proof is more complicated. First we use an
observation due to R. Schoof:

\begin{prop}\label{PS}
Let $L/k$ be a normal extension with Galois group $\Gamma$,
and suppose that $3 \nmid \# \Gamma$. Let $K$ be the maximal
abelian extension of $k$ contained in $L$. If $\Gamma$
does not have a quotient of type $SD_{16}$, 
if $\Cl_3(L) \simeq (3,3)$, and if $3 \nmid h(k)$, then
there exists a subfield $E$ of $L/k$ such that 
$\Cl_3(E) \simeq (3,3)$ and $\Gal(E/k) \simeq D_4$ or $H_8$.
\end{prop}

\begin{proof}
Put $A = \Cl_3(L)$; then $\Gamma$ acts on $A$, i.e. there is
a homomorphism $\Phi: \Gamma \lra \Aut(A) \simeq \GL(2,3)$. 
From $\# \GL(2,3) = 48$ and $3 \nmid \Gamma$ we conclude that
$\im \Phi$ is contained in the $2$-Sylow subgroup of $\GL(2,3)$,
which is $\simeq SD_{16}$. 

We claim that $\im \Phi$ is not abelian. In fact, assume that
it is. Then $\im \Phi \simeq \Gamma/\ker \Phi$ shows that
$\Gamma' \subseteq \ker \Phi$; hence $\Gamma'$ acts trivially
on $A$. The fixed field of $\Gamma'$ is $K$, and now
$ \Cl_3(K) \supseteq N_{L/K} \Cl_3(L) = A^{(L:K)} \simeq A$
shows that $3 \mid h(K)$ contradicting our assumptions.

Thus $\im \Phi$ is a nonabelian $2$-group, and we conclude
that $\# \im \Phi \ge 8$. On the other hand, $\Gamma$ does
not have $SD_{16}$ as a quotient, hence $\# \im \Phi = 8$,
and we have $\im \Phi \simeq D_4$ or $H_8$, since these are the 
only nonabelian groups of order $8$. Let $E$ be the fixed
field of $\ker \Phi$. Since $\ker \Phi$ acts trivially on $A$,
we see $\Cl_3(E) \supseteq N_{L/E} \Cl_3(L) = A^{(L:E)} \simeq A$
(note that $(L:E) \mid \Gamma$ is not divisible by $3$). The other
assertions are clear.
\end{proof}

Applying the proposition to the extension $k^2/k$ (we have to 
check that $\Gal(k^2/\Q)$ does not have $SD_{16}$ as a quotient.
But $\Gal(k^2/\Q)$ is a group extension
$$ 1 \lra \Gal(k^2/k^1) \lra \Gal(k^2/\Q) \lra \Gal(k^1/\Q) \lra 1$$
of an elementary abelian group $\Gal(k^1/\Q) \simeq (2,2,2,2)$
by another elementary abelian group $\Gal(k^2/k^1) \simeq (2,2)$,
from which we deduce that $\Gal(k^2/\Q)$ has exponent $4$. Now
observe that $SD_{16}$ has exponent $8$) we find that
$3 \nmid h(k^2)$ unless $k^2$ contains a normal extension
$E/\Q$ with $3$-class group of type $(3,3)$ and Galois group
isomorphic to $D_4$ or $H_8$. Let $E_0$ be the maximal abelian 
subfield of $E$; this is a $V_4$-extension of $\Q$ with quadratic
subfields $k_0$, $k_1$ and $k_2$. Let
$k_0$ denote a quadratic subfield over which $E$ is cyclic. 
If a prime ideal $\fp$ ramifies
in $E_0/k_0$, then it must ramify completely in $E/k_0$;
but since all prime ideals have ramification index $\le 2$ in $E$
(since $E \subset k^2$), this is a contradiction.

Thus $E/k_0$ is unramified. If we put $d_j = \disc k_j$, then
this happens if and only if $d_0 = d_1d_2$ and $(d_1,d_2) = 1$. 
For quaternion extensions, this is already a contradiction, 
since we conclude by symmetry that $d_1 = d_0d_2$ and 
$d_2 = d_0d_1$. Assume therefore that $\Gal(E/\Q) \simeq D_4$. 
From a result of Richter \cite{R} we know that $E_0$ can only 
be embedded into a dihedral extension if $(d_1/p_2) = (d_2/p_1) = +1$
for all $p_1\mid d_1$ and all $p_2\mid d_2$. But $E_0$
is an abelian extension contained in $k^2$, hence it is
contained in $k_\gen$, and we see that $d$ is a product
of the prime discriminants $-3$, $-4$, $5$, and $-7$. Since
no combination of these factors satisfies Richter's conditions,
$E_0$ cannot be embedded into a $D_4$-extension $E/\Q$.
This contradiction concludes our proof that $3 \nmid h(k_2)$.

\section*{Acknowledgement}
Much of Section \ref{Secl}, in particular Proposition \ref{PS},
is due to R. Schoof \cite{Sch}, whom 
I would like to thank for his help. I also thank C. Snyder
for correcting an error in Table \ref{T2}.

\end{document}